\documentclass[11pt,reqno,oneside]{amsart}

\usepackage[T1]{fontenc}
\usepackage[utf8x]{inputenc}
\usepackage[english]{babel}

\usepackage{mathtools,amssymb,amsthm}

\usepackage[foot]{amsaddr}

\makeatletter
\renewcommand{\email}[2][]{%
  \ifx\emails\@empty\relax\else{\g@addto@macro\emails{,\space}}\fi%
  \@ifnotempty{#1}{\g@addto@macro\emails{\textrm{(#1)}\space}}%
  \g@addto@macro\emails{#2}%
}
\makeatother

\usepackage{enumitem,multirow}
\setenumerate{label={{\rm({\roman*})}},leftmargin=6mm,itemsep=3pt,topsep=3pt}
\setitemize{label={$\vcenter{\hbox{\tiny$\bullet$}}$},leftmargin=6mm,itemsep=3pt,topsep=3pt}

\usepackage[margin=2.75cm]{geometry}
\usepackage{xcolor}

\renewcommand{\le}{\leqslant}
\renewcommand{\ge}{\geqslant}
\renewcommand{\leq}{\leqslant}
\renewcommand{\geq}{\geqslant}
\renewcommand{\phi}{\varphi}
\newcommand{\defeq}{\coloneqq}
\renewcommand{\emptyset}{\varnothing}

\newcommand{\Z}{\mathbb{Z}}
\newcommand{\R}{\mathbb{R}}
\newcommand{\M}{\mathcal{M}}

\newcommand{\mtra}{^{\raisebox{-.45ex}{\scalebox{.7}[1.0]{$-$}}\mkern-1mu\intercal}}
\newcommand{\tra}{^{\mkern-.2mu\intercal}}

\newcommand{\ip}[2]{\langle{#1},{#2}\rangle}
\newcommand{\set}[2]{\bigl\{#1 \bigm\vert #2\bigr\}}

\newcommand{\0}{\boldsymbol{0}}
\newcommand{\ee}{\boldsymbol{e}}

\newcommand{\conv}{\mathrm{conv}}
\newcommand{\stab}{\mathrm{STAB}}

\newcommand{\rk}{\mathrm{rk}}

\newcommand{\cone}{\mathop{\mathrm{cone}}}
\newcommand{\polylog}{\mathop{\mathrm{polylog}}}

\newtheorem{thm}{Theorem}

\newtheorem{lem}[thm]{Lemma}
\newtheorem{clm}[thm]{Claim}

\newtheorem{cor}[thm]{Corollary}
\newtheorem{rem}[thm]{Remark}

\theoremstyle{definition}

\makeatletter
\newtheorem*{rep@theorem}{\rep@title}
\newcommand{\newreptheorem}[2]{%
\newenvironment{rep#1}[1]{%
 \def\rep@title{#2\ \ref{##1}}%
 \begin{rep@theorem}\itshape}
 {\end{rep@theorem}}}
\makeatother

\newreptheorem{thm}{Theorem}

\begin{document}

\title{Bounds on the number of 2-level polytopes, cones and configurations}


\author{Samuel Fiorini\textsuperscript{1}}
\author{Marco Macchia\textsuperscript{2}}
\address[1,2]{Universit\'{e} libre de Bruxelles}
\email[1,2]{\{sfiorini,mmacchia\}@ulb.ac.be}

\author{Kanstantsin Pashkovich\textsuperscript{3}}
\address[3]{C \& O Department, University of Waterloo, Waterloo, Canada}
\email[3]{kpashkov@uwaterloo.ca}

\begin{abstract}
We prove an upper bound of the form $2^{O(d^2 \polylog d)}$ on the number of affine (resp.\ linear) equivalence classes of, by increasing order of generality, $2$-level $d$-polytopes, $d$-cones and $d$-configurations. This in particular answers positively a conjecture of Bohn et al.\ on $2$-level polytopes~\cite{Bohn17}. We obtain our upper bound by relating affine (resp.\ linear) equivalence classes of $2$-level $d$-polytopes, $d$-cones and $d$-configurations to faces of the correlation cone.

We complement this with a $2^{\Omega(d^2)}$ lower bound, by estimating the number of nonequivalent stable set polytopes of bipartite graphs. 
\end{abstract}

\maketitle

\section{Introduction}

For a positive integer $d$, consider the set $\M_d$ of $0/1$-matrices with rank-$d$ and no repeated row or column. How many matrices in $\M_d$ are \emph{maximal}, in the sense that they are not a submatrix of another matrix in $\M_d$? We show that $\M_d$ has surprisingly few maximal elements up to permutations of rows and columns, see Theorem~\ref{thm:refined_upper_bound_slack_matrices}. 
This is achieved by parametrizing these matrices by graphs on $n = O(d \log d)$ nodes, whose edges have integer weights in $[0,n^2]$.

\subsection{Motivation}

Matrices with $0/1$-entries are studied in communication complexity. This field is concerned with quantifying the amount of information that two or more parties have to exchange in order to collaboratively evaluate a function~\cite{Kushilevitz96}.

We consider the simplest communication model with two parties, that of deterministic protocols~\cite{Yao79}. It is well known that the \emph{deterministic communication complexity} $\mathrm{CC}(f)$ of a Boolean function $f : X \times Y \to \{0,1\}$ can be lower bounded by the base-$2$ logarithm of the rank of its \emph{communication matrix} $M(f)$:
\[
\mathrm{CC}(f) \geqslant \log\rk(M(f))\,.
\]
A fundamental open problem known as the \emph{log-rank conjecture}~\cite{LovaszSaks93} states that this simple bound is almost tight for all Boolean functions $f$, in the sense that
\[
\mathrm{CC}(f) \leqslant \polylog \rk(M(f))\,.
\]

Proving the log-rank conjecture in the special case where $M(f)$ is a maximal element of $\M_d$ with $d \coloneqq \rk(M(f))$ is enough to imply the whole conjecture, because deterministic communication complexity is monotone under taking submatrices, and invariant under repeating rows or columns. In this sense, maximal elements of $\M_d$ are the hardest instances for the log-rank conjecture. To our knowledge, this fact has not been used in the communication complexity literature. Our hope is that the succinct representation of such maximal matrices as weighted graphs introduced in this paper can be used in a new approach to the log-rank conjecture. 

Matrices with $0/1$-entries also appear in discrete geometry and discrete optimization. Let $P = \set{x \in \R^d}{Ax \ge b} = \conv(V)$ be a polytope, where $A\in\R^{m\times d}$ has $m$ rows denoted by $A_i$, $i \in [m] \coloneqq \{1,\ldots,m\}$ and $V\subseteq \R^d$ has $n$ points denoted by $v_j$, $j \in [n] = \{1,\ldots,n\}$. The $m \times n$ nonnegative matrix $S$ whose $(i,j)$-entry is the slack of the $j$-th vertex $v_j$ with respect to the $i$-th inequality $A_i x \ge b_i$, that is,  $S_{ij} \coloneqq A_i v_j - b_i$, is referred to as a \emph{slack matrix} of $P$. A similar definition holds for polyhedral cones.

A polytope $P$ (or polyhedral cone $K$) is \emph{$2$-level} provided that it admits a slack matrix with $0/1$-entries~\cite{Gouveia10}. 

Some examples of $2$-level polytopes in the literature are Hanner polytopes~\cite{Hanner56}, Birkhoff polytopes~\cite{Birkhoff46}, or more generally, polytopes of the form $P = \set{x \in [0,1]^d}{Ax = b}$ for some $A$ totally unimodular\footnote{A matrix $M$ is said to be \emph{totally unimodular} provided that the determinant of every square submatrix of $M$ is either $0$,$1$ or $-1$, see for instance~\cite{Schrijver86}.} and $b$ integral, order polytopes~\cite{Stanley86}, and stable set polytopes of perfect graphs~\cite{Chvatal75}.

It is an open problem to determine what function of $d$ describes the number of affine equivalence classes of $2$-level $d$-polytopes.
In~\cite{Bohn17}, it is conjectured that this number is at most $2^{\mathrm{poly}(d)}$. This conjecture is backed by experimental evidence: Bohn et al.~\cite{Bohn17} enumerate all affine equivalence classes of $2$-level $d$-polytopes for dimension $d \leq 7$. (We point out that the enumeration algorithm has been since then improved\footnote{The latest implementation of the code is due to Samuel Fiorini, Marco Macchia, Aur\'elien Ooms (Universit\'e libre de Bruxelles). The source code and the complete list of all slack matrices of non-isomorphic $2$-level polytopes up to dimension $8$ are available online at \texttt{https://github.com/ulb/tl}\,.} and produced the complete database up to $d = 8$.)

Further related work is that of Grande and Ru{\'e}~\cite{Grande15b}, who provide a $O(c^d)$ lower bound on the number of 2-level matroid $d$-polytopes, for some constant $c > 0$. Finally, in~\cite{Gouveia16}, Gouveia \emph{et al.} completely classify polytopes with minimum positive semidefinite rank (which generalize $2$-level polytopes) in dimension $d=4$.

We point out that if the log-rank conjecture holds, every $2$-level $d$-polytope can be described as the projection of a polytope with at most $2^{\polylog(d)}$ facets. This implication follows from a classic result of Yannakakis~\cite{Yannakakis91} linking the \emph{extension complexity} of a polytope to the \emph{nonnegative rank} of its slack matrices. It is known that the extension complexity of the stable set polytopes of a perfect graph is at most $2^{\polylog(d)}$~\cite{Yannakakis91}, but whether or not the same proof strategy can be generalized to the entire class of $2$-level $d$-polytopes is still open.

\subsection{Contribution and outline}

We phrase the counting problem for $0/1$-matrices in terms of counting so called \emph{$2$-level configurations}, that we formally define at the beginning of Section~\ref{sec:preliminary}. Basically, a $2$-level configuration is a rank factorization of maximal matrix in $\mathcal{M}_d$. These configurations also capture (maximal) slack matrices of $2$-level cones and $2$-level polytopes.

In Section~\ref{sec:preliminary}, we introduce the notion of \emph{linear equivalence} for $2$-level configurations. Intuitively, two $2$-level configurations are linearly equivalent if they are two rank factorizations of the same matrix in $\mathcal{M}_d$. Moreover, we show that given some $2$-level $d$-polytope $P$, we can associate it to a $2$-level $(d+1)$-configuration. Similarly, a $2$-level $d$-cone $K$ can be associated to a $2$-level $d$-configuration.

In Section~\ref{sec:lower_bound}, we present a lower bound of $2^{\Omega(d^2)}$ for the number of affine equivalence classes of $2$-level $d$-polytopes, that implies a lower bound for the number of linear equivalence classes of $2$-level $d$-configurations.

Moreover, in Section~\ref{sec:naive_upper_bound}, we prove a first upper bound of $2^{O(d^3)}$ for the number of linear equivalence classes of $2$-level $d$-configurations.
Next, in Section~\ref{sec:upper_bound_cones} we present our first main result:

\begin{thm}\label{thm:second_upper_bound}
The total number of affine equivalence classes of $2$-level $d$-polytopes and the total number of linear equivalence classes of $2$-level $d$-cones is at most~$2^{O(d^2\log d)}$.
\end{thm}
This theorem follows from upper bounding the number of faces of the \emph{correlation cone}~\cite[Chapter 5]{Deza09}. 

Finally, in Section~\ref{sec:improved_upper_bound} we refine the $2^{O(d^3)}$ bound of Section~\ref{sec:naive_upper_bound} and prove our second main result:

\begin{thm}\label{thm:refined_upper_bound}
The total number of linear equivalence classes of $2$-level $d$-configurations is at most~$2^{O(d^2\log^3(d))}$.
\end{thm}

\section{Preliminary definitions and results}
\label{sec:preliminary}

A (convex) polytope $P\subseteq \R^d$ is said to be \emph{$2$-level} if there is a finite system of linear inequalities\footnote{Throughout, we assume that systems of linear inequalities do not have repeated inequalities.} $Ax \geqslant b$ with $A \in \R^{m \times d}$, $b \in \R^m$ and a finite point set $V \subseteq \R^d$ such that
\begin{align}
\label{eq:descr}
&P = \set{x \in \R^d}{Ax \ge b} = \conv(V)\,, \\ \text{ and}\\
\label{eq:2-level}
&A_i v_j-b_i \in \{0,1\} \text{ for all } i\in [m] \text{ and } j\in [n]\,,
\end{align}
where $A_i$ denotes the $i$-th row of $A$ and $v_j$ the $j$-th point of $V$.

Let $P \subseteq \R^d$ be a full-dimensional $2$-level polytope. We call a pair $(Ax \ge b, V)$ \emph{a maximal pair of $2$-level descriptions for $P$} if $Ax \ge b$ and $V$ satisfy \eqref{eq:descr} and \eqref{eq:2-level} and are maximal with respect to these properties. Notice that the inequalities $\ip{\0_d}{x} \ge 0$ and $\ip{\0_d}{x} \ge -1$ are always part of $Ax \ge b$ if $(Ax \ge b, V)$ is a maximal pair of $2$-level descriptions for $P$. Since we assume that $P$ is full dimensional, every $2$-level polytope admits a maximal pair of $2$-level descriptions. We call a matrix $S = S(Ax \ge b, V) \in \R_{\ge 0}^{m \times n}$ a \emph{maximal slack matrix} of $P$ if $S_{ij} = A_i v_j - b_i$ for some maximal pair $(Ax \ge b, V)$ of $2$-level descriptions for $P$, where $A$ has $m$ rows and $V$ has $n$ points. 

Likewise, a (pointed, polyhedral) cone $K \subseteq \R^d$ is said to be \emph{$2$-level} if there is a finite system of homogeneous linear inequalities $Ax \ge \0$ with $A \in \R^{m \times d}$ and a finite vector set $V \subseteq \R^d$ such that $K = \set{x\in\R^d}{Ax \geqslant \0} = \cone(V)$ and $A_i v_j \in \{0,1\}$ for all $i\in [m]$ and $j\in [n]$. The notions of maximal pair of $2$-level descriptions and maximal slack matrix are defined similarly for cones as for polytopes. As in the case of polytopes, we assume cones to be full dimensional. This guarantees the existence of a maximal pair of $2$-level descriptions and a maximal slack matrix for each $2$-level cone. Although it plays no role in this paper, we remark that maximal slack matrices are unique for $2$-level polytopes, but not necessarily for $2$-level cones.

A \emph{$2$-level $d$-configuration} is a pair $(A,B)$, where $A\subseteq \R^d$ and $B\subseteq \R^d$ are two maximal sets of vectors linearly spanning $\R^d$ such that the inner product $\ip{a}{b} = a\tra b$ is in $\{0,1\}$ for all $a \in A$ and $b \in B$. 
Letting $S_{ab} \coloneqq  \ip{a}{b}$, we obtain a matrix $S = S(A,B) \in \R_{\ge 0}^{|A| \times |B|}$ that is the \emph{slack matrix} of the $2$-level $d$-configuration $(A,B)$. Two $2$-level $d$-configurations $(A,B)$ and $(A',B')$ are called \emph{linearly equivalent} if there exists a non-singular matrix $M\in \R^{d\times d}$ such that 
\[
A' = \set{M\mtra a}{a\in A}
\quad\text{and}\quad
B' = \set{M b}{b\in B}\,.
\]

The next remark follows from the fact that, for a $2$-level $d$-configuration $(A,B)$, the sets $A$ and $B$ are required to linearly span $\R^d$. 

\begin{rem}
For a $2$-level $d$-configuration $(A,B)$, the slack matrix $S(A,B)$ is a maximal element in $\M_d$.
\end{rem}

Let us first observe that $2$-level polytopes and $2$-level cones can be interpreted as special instances of $2$-level configurations.

\begin{lem}\label{lem:2L_cone_conf}
Consider a full-dimensional $2$-level cone $K = \set{x\in\R^d}{Ax \geqslant \0} = \cone(V)$, where $A\in \R^{m\times d}$, $V = \{v_1,\ldots,v_n\}$ and $(Ax \geqslant \0, V)$ is a maximal pair of $2$-level descriptions for $K$. Then the pair
\[
A'\coloneqq \set{A_i\tra }{i\in[m]}\quad\text{and}\quad B'\coloneqq \set{v_j}{j\in[n]}
\]
defines a $2$-level $d$-configuration.
\end{lem}

\begin{proof}
Since $K$ is full-dimensional, it is enough to prove that both $A'$ and $B'$ are maximal. This holds because the set of linear inequalities $Ax \ge \0$ and the set $V$ are maximal with respect to the property $A_iv_j \in \{0,1\}$, where $A_i$ is the $i$-th row of the matrix $A$ and $v_j$ is the $j$-th vector in $V$ for $i \in [m]$, $j \in [n]$.
\end{proof}

\begin{lem} \label{lem:2L_poly_conf}
Consider a full-dimensional $2$-level polytope $P = \set{x \in \R^d}{Ax \ge b} = \conv(V)$, where $A\in \R^{m\times d}$, $b\in \R^m$, $V = \{v_1,\ldots,v_n\}$ and $(Ax \ge b, V)$ is a maximal pair of $2$-level descriptions for $P$. Then the pair
\[
A'\coloneqq \set{(A_i\tra ,b_i)}{i\in[m]}\quad\text{and}\quad B'\coloneqq \set{(v_j,-1)}{j\in[n]}\cup\{\0_{d+1}\}
\]
defines a $2$-level $(d+1)$-configuration.
\end{lem}

\begin{proof}
To show that $(A',B')$ is a $2$-level $(d+1)$-configuration, it is enough to show that both $A'$ and $B'$ are maximal. Since the system $Ax \ge b$ is maximal with respect to the property $A_iv-b_i \in \{0,1\}$ for all $i$, $j$, we obtain that $A'$ is maximal. Also, as we noticed before, the inequality $\ip{\0_{d}}{x}\geq -1$ is present in $Ax \ge b$.

In order to prove that $B'$ is maximal, we have to show that there exists no $u=(v,t)\in \R^{d+1}$ with $v\in \R^d$, $t\in \R$, $t\neq -1$ and $u\neq \0_{d+1}$, such that $\ip{a'}{u}\in\{0,1\}$ for all $a'\in A'$. Let us assume that there exists such a vector $u$.

From $(\0_{d},-1) \in A'$ we get that $t\in \{0,-1\}$ because $-t=\ip{(\0_{d},-1)}{u}\in\{0,1\}$.
Since $t\in \{0,-1\}$ and $t\neq -1$, we have $t=0$. Thus, $\ip{a'}{u}\in\{0,1\}$ for all $a'\in A'$ implies that $Av \in \{0,1\}^m$. In particular, $v$ is in the recession cone\footnote{Given a nonempty convex set $C \subseteq \R^d$, the \emph{recession cone} of $C$ is the set of all directions along which we can move indefinitely and still be in $C$, i.e. $\set{y \in \R^d}{x+\lambda y \in C,\,\forall x\in C,\,\forall \lambda \ge 0}$\,.} of $P$. Since $P$ is a polytope, its recession cone contains no vector besides $\0_d$. Hence, $v=\0_d$ and $u=\0_{d+1}$, a contradiction. 
\end{proof}

An alternative proof of Lemma~\ref{lem:2L_poly_conf} uses Lemma~\ref{lem:2L_cone_conf} and the fact that every $2$-level $d$-polytope naturally yields a $2$-level $(d+1)$-cone pointed at the origin. In fact, let $P$ be a $2$-level $d$-polytope $P = \set{x \in \R^d}{Ax \geqslant b} = \conv(V)$, where $A\in \R^{m\times d}$, $b\in \R^m$, $V = \{v_1,\ldots,v_n\}$ and $(Ax \geqslant b, V)$ is a maximal pair of $2$-level descriptions for $P$.
Let $A_i' \defeq (A_i\tra,b_i)$ for every $i \in [m]$ and $v_j' \defeq (v_j,-1)$ for every $j \in [n]$, $v'_{n+1}  \defeq \0_{d+1}$. Then it is straightforward to check that $K \defeq \cone(V') = \set{x \in \R^d}{A'x \ge \0}$ is a $2$-level $(d+1)$-cone pointed at the origin, where $V' \defeq \{v_1',\ldots,v_{n+1}'\}$ and $A'$ is the matrix whose rows are $A_i'$, $i \in [m]$. Moreover, $(A'x \ge \0,V')$ is a maximal pair of $2$-level descriptions for $K$.

We now show that $2$-level polytopes, cones, or configurations can be encoded using their maximal slack matrices, i.e.\ they can be encoded as some $0/1$-matrices. We show that this a valid encoding, i.e.\ given a $0/1$ maximal slack matrix we can reconstruct the original $2$-level polytope up to affine transformation, or reconstruct the $2$-level cone or $2$-level configuration up to linear transformation. We would like to note that the analogous statement holds for general polytopes and general pointed cones and slack matrices, which are not necessarily maximal.

\begin{lem}
If two $2$-level $d$-configurations $(A,B)$, $(A',B')$ admit the same slack matrix up to permutation of rows and columns, then $(A,B)$ and $(A',B')$ are linearly equivalent.
\end{lem}

\begin{proof}
Without loss of generality, we assume  $S(A,B) = S(A',B')$.
Both $(A,B)$ and $(A',B')$ provide rank factorizations of the matrix $S(A,B) = S(A',B')$ with rank $d$. Thus, there exists a non-singular matrix $M\in \R^{d\times d}$ such that \[
A'\coloneqq \set{M\mtra a}{a\in A}
\quad\text{and}\quad
B'\coloneqq \set{M b}{b\in B}\,.
\]
The claim follows.
\end{proof}

\begin{cor}
If two $2$-level polytopes admit the same maximal slack matrix up to permutation of rows and columns, then these polytopes are affinely equivalent. Similarly, if two $2$-level cones admit the same maximal slack matrix up to permutation of rows and columns, then these cones are linearly equivalent.
\end{cor}

The next lemma shows that given a $2$-level $d$-configuration $(A,B)$, we can find another $2$-level $d$-configuration $(A',B')$ with the same maximal slack matrix and where $A'$ is a set of $0/1$-vectors.

\begin{lem}\label{lem:01-configuration-A}
Given a $2$-level $d$-configuration $(A,B)$, there exists a $2$-level $d$-configuration $(A',B')$ such that $A' \subseteq \{0,1\}^d$ and $S(A,B) = S(A',B')$.
\end{lem}
\begin{proof}
Since $B$ linearly spans $\R^d$, there are $d$ vectors $b_1, b_2, \ldots, b_d\in B$ which linearly span $\R^d$. Let $M\in \R^{d\times d}$ be the matrix whose $i$-th column equals $b_i$ for $i\in [d]$. Then the set
\[
B'\coloneqq \set{M^{-1} b}{b\in B}\,,
\]
contains standard basis vectors $\ee_1,\ldots, \ee_d$.

Moreover, the $2$-level $d$-configuration $(A',B')$, where \[
A'\coloneqq \set{M\tra a}{a\in A}\,,
\]
has the same maximal slack matrix as the $2$-level $d$-configuration $(A,B)$. 

Finally, since $S(A',B')$ is a $0/1$-matrix and the set $B'$ contains standard basis vectors $\ee_1,\ldots, \ee_d$, we have that $A'$ is a set of $0/1$-vectors.
\end{proof}

Analogously one can show the following lemma:

\begin{lem}\label{lem:01-configuration-B}
Given a $2$-level $d$-configuration $(A,B)$, there exists a $2$-level $d$-configuration $(A',B')$ such that $B' \subseteq \{0,1\}^d$ and $S(A,B) = S(A',B')$.
\end{lem}

Lemma~\ref{lem:01-configuration-A} can be strengthened for $2$-level configurations arising from $2$-level polytopes and $2$-level cones.

We recall the definition of simplicial core for polytopes, that appeared in \cite[Definition 1]{Bohn17}, and generalize it to cones.
A \emph{simplicial core} for a $d$-polytope $P$ is a $(2d+2)$-tuple $(F_1,\ldots,F_{d+1};$ $v_1,\ldots,v_{d+1})$ of facets and vertices of $P$ such that each facet $F_i$ does not contain vertex $v_i$ but contains vertices $v_{i+1}$, \ldots, $v_{d+1}$.
A \emph{simplicial core} for a $d$-cone $K$ is a $2d$-tuple $(F_1,\ldots,F_d;$ $v_1,\ldots,v_d)$ of facets and extreme rays of $K$ such that each facet $F_i$ does not contain extreme ray $v_i$ but contains extreme rays $v_{i+1}, \ldots, v_d$. The proof of next result follows from \cite[Lemma 9 and  Corollary 10]{Bohn17}. For the sake of completeness, we present another version here.

\begin{lem}\label{lem:01-configuration-polytope}
Given a $2$-level $d$-polytope $P$, there exists a $2$-level $(d+1)$-configuration $(C,D)$ such that $C\subseteq \{0,1\}^{d+1}$, $D\subseteq \Z^{d+1}$ and $S(C,D)$ is a maximal slack matrix of $P$.
\end{lem}

\begin{proof}
Let $(Ax \ge b,V)$ define a maximal pair of $2$-level descriptions for $P$, where $A \in \R^{m \times d}$ and $V \coloneqq \{v_1,\dots,v_n\}$.
By~\cite[Proposition 3.2]{GouvRobThom13}, \cite[Lemma 2]{Bohn17}, there exist $d+1$ facets $F_1,\ldots,F_{d+1}$ and $d+1$ vertices $v_1,\ldots,v_{d+1}$ such that $\Gamma \coloneqq (F_1,\ldots,F_{d+1};v_1,\ldots,v_{d+1})$ is a simplicial core for $P$. Let us assume that the facets $F_1,\ldots,F_{d+1}$ are defined by the linear inequalities $A_1 x \ge b_1, \ldots, A_{d+1} x \ge b_{d+1}$ respectively. By Lemma~\ref{lem:2L_poly_conf}, the pair of sets
\[
A'\coloneqq \set{(A_i\tra ,b_i)}{i\in[m]}\quad\text{and}\quad B'\coloneqq \set{(v_j,-1)}{j\in[n]}\cup\{\0_{d+1}\}
\]
defines a $2$-level $(d+1)$-configuration.

Define $M\in \R^{(d+1)\times(d+1)}$ to be the matrix whose $i$-th row is given by $M_i=(A_i, b_i)$. Due to the definition of simplicial core, the matrix $M$ is non-singular. Now, let
\[
A''\coloneqq \set{M\mtra a'}{a'\in A'}
\quad\text{and}\quad
B''\coloneqq \set{M b'}{b'\in B'}\,.
\]
Clearly, $(A'', B'')$ is a $2$-level $(d+1)$-configuration. Moreover, $B''\subseteq \{0,1\}^{d+1}$, since $(A',B')$ is a $2$-level configuration.

Let $b''_1,\ldots, b''_{d+1}$ be the vectors in $B''$ corresponding to the vertices $v_1,\ldots,v_{d+1}$ of $P$, respectively. Let $L\in \R^{(d+1)\times(d+1)}$ be the matrix with $i$-th column equal to $b''_i$. Since $\Gamma$ is a simplicial core, $L$ is a non-singular lower-triangular $0/1$ matrix, thus $L$ is unimodular.

Define
\[
C\coloneqq \set{L\tra a''}{a''\in A''}
\quad\text{and}\quad
D\coloneqq \set{L^{-1} b''}{b''\in B''}\,.
\]
As before, $(C, D)$ is a $2$-level $(d+1)$-configuration.
Since $L$ is a unimodular $0/1$-matrix and $B''\subseteq \{0,1\}^{d+1}$, we have that all vectors in $D$ are integral. Moreover, vectors $\ee_1,\ldots,\ee_{d+1}$ are in $D$, hence all vectors in $C$ are $0/1$-vectors. We conclude that $(C,D)$ is the desired $2$-level $(d+1)$-configuration.
\end{proof}

A result analogous to Lemma~\ref{lem:01-configuration-polytope} holds for $2$-level cones. Its proof follows the proof of the previous lemma and is left to the reader.

\begin{lem}\label{lem:01-configuration-cone}
Given a $2$-level $d$-cone $K$, there exists a $2$-level $d$-configuration $(C,D)$ such that $C\subseteq \{0,1\}^d$, $D\subseteq \Z^d$ and $S(C,D)$ is a maximal slack matrix for $K$.
\end{lem}

\section{Lower bound on the number of 2-level polytopes}
\label{sec:lower_bound}

In this section we prove that the number of affine equivalence classes of $2$-level $d$-polytopes is $2^{\Omega(d^2)}$. To do that we use a well known family of $2$-level polytopes: the family of stable set polytopes of bipartite graphs. 
First, we show that two non-isomorphic bipartite graphs lead to affinely nonequivalent stable set polytopes, whenever the minimum degree of both graphs is at least $2$. 
Then, we use the result by~\cite{ErdKlei73}  to show the lower bound $2^{\Omega(d^2)}$ for the number of isomorphism classes for bipartite graphs with minimum degree at least $2$.

 The \emph{stable set polytope} of a graph $G=(V,E)$, denoted by $\stab(G)$, is the convex hull of the characteristic vectors of stable sets of $G$.
Let $G=(V,E)$ be a bipartite graph with no isolated nodes, then the stable set polytope $\stab(G)$ can be described in the following way:
\begin{equation}\label{eq:stab_bip}
\stab(G) = 
\biggl\{
\begin{array}{@{\hskip.5mm}c@{\hskip1mm}c@{\hskip1mm}l@{\hskip2mm}l@{\hskip.5mm}}
\multirow{2}{*}{$x \in \R^V$} & \multirow{2}{*}{$\biggm\vert$} & x_v \ge 0  &\text{for all } v \in V \\
&& x_u + x_v \le 1 &\text{for all } \{u,v\} \in E
\end{array}
\biggr\}\,.
\end{equation}
It is straightforward to verify that each of the above inequalities defines a facet of $\stab(G)$ whenever $G$ is bipartite. Moreover, since $\stab(G)$ is a full-dimensional polytope, different inequalities above define different facets of $\stab(G)$.

\begin{clm}\label{clm:unique_simple_vertex}
Let  $G=(V,E)$ be a $d$-node graph such that the minimum degree of a node in $G$ is at least~$2$. Then $\stab(G)$  has a unique simple vertex $\0_d$, i.e.~$\0_d$ is the only vertex of $\stab(G)$ contained in exactly $d$ facets of $\stab(G)$.
\end{clm}
\begin{proof}
 Let us consider a vertex $w$ of $\stab(G)$, corresponding to a stable set $S\subseteq V$. The vertex $w$ is contained in some facets defined by non-negativity constraints and some facets defined by edge constraints. First, there are exactly $|V\setminus S|=d-|S|$ facets corresponding to non-negativity constraints, which contain the vertex $w$. Second, to each edge incident to a node in $S$ corresponds a facet which contains $w$. Since each node in $G$ has degree at least $2$, there are at least
\[
\big(d-|S|\big)+2|S|=d+|S|
\]
facets containing the vertex $w$. Now, note that the vertex $\0_d$ of $\stab(G)$ corresponding to $S=\emptyset$ is contained in exactly $d$ facets, finishing the proof.
\end{proof}

\begin{clm}\label{clm:neighbours_empty_set}
Given a $d$-node graph $G=(V,E)$, the vertex $\0_d$ is incident only to the vertices $\ee_v$, $v\in V$ of $\stab(G)$.
\end{clm}

\begin{proof}
The statement follows from the fact that the vertex $\0_d$ of $\stab(G)$ is incident only to the facets induced by non-negativity constraints $x_v\ge 0$, $v\in V$. 
\end{proof}

\begin{lem}\label{lem:non-isom_bip_stab}
Let $G_1 = (V_1,E_1)$, $G_2 = (V_2,E_2)$ be two graphs such that the minimum degree of nodes in $G_1$ and $G_2$ is at least $2$. Then $G_1$, $G_2$ are isomorphic if and only if $\stab(G_1)$, $\stab(G_2)$ are affinely equivalent.
\end{lem}

\begin{proof}
Clearly, if $G_1$ and $G_2$ are isomorphic then $\stab(G_1)$, $\stab(G_2)$ are affinely equivalent as well.

Now suppose that $\stab(G_1)$, $\stab(G_2)$ are affinely equivalent, i.e.\ there exists a bijective affine map $\mu:\R^d\rightarrow \R^d$ such that $\mu (\stab(G_1))=\stab(G_2)$, where $d=|V_1|=|V_2|$. By Claim~\ref{clm:unique_simple_vertex}, we have $\mu (\0_d)=\0_d$, and by Claim~\ref{clm:neighbours_empty_set}, we have $\mu(\ee_v)=\ee_{\pi(v)}$ for all $v\in V_1$, where $\pi:V_1\rightarrow V_2$ is a bijection. Now, it is straightforward to verify that $\pi$ defines an isomorphism between $G_1$ and $G_2$, because for a graph $G=(V, E)$ we have that $\{u,v\} \in E$ if and  only if $\0_d$, $\ee_u$, $\ee_v$ define a triangular face of $\stab(G)$.
\end{proof}

Let $V$ be any set of cardinality $d$. The result by~\cite[Lemma 7]{ErdKlei73} shows that the number of labeled bipartite graphs with node set $V$ is at least $2^{\frac{d^2}{4}+d-2 \log d}$ and at most $2^{\frac{d^2}{4}+d}$. Hence, 
\[
d^2 2^{\frac{(d-1)^2}{4}+d}=2^{\frac{d^2}{4}+\frac{d}{2}+\frac{1}{4}+2\log{d}}
\] 
is an upper bound for the number of labeled bipartite graphs with node set $V$ and minimum degree less than $2$. Thus there are at least
\[
2^{\frac{d^2}{4}+d-2 \log d} - 2^{\frac{d^2}{4}+\frac{d}{2}+\frac{1}{4}+2\log{d}}
\]
many labeled bipartite graphs with node set $V$ and minimum degree at least $2$. Therefore, there are $2^{\frac{d^2}{4}(1-o(1))} = 2^{\Omega(d^2)}$ isomorphism classes of bipartite graphs on $d$ nodes with minimum degree at least $2$. This fact together with Lemma~\ref{lem:non-isom_bip_stab} implies the following result.

\begin{thm}
The number of affine equivalence classes of $2$-level $d$-polytopes is $2^{\Omega(d^2)}$.
\end{thm}

By Lemma~\ref{lem:2L_poly_conf}, every $2$-level $d$-polytope yields a $2$-level $(d+1)$-configuration. We obtain the following result:

\begin{cor}
The number of linear equivalence classes of $2$-level $d$-configurations is $2^{\Omega(d^2)}$.
\end{cor}

\section{First upper bound for 2-level configurations}
\label{sec:naive_upper_bound}

The next theorem gives a first upper bound on the total number of maximal $0/1$-matrices with rank $d$. This will be the basis for the refined upper bounds in Sections~\ref{sec:upper_bound_cones} and~\ref{sec:improved_upper_bound}.

\begin{thm}\label{thm:naive_upper_bound}
The number of maximal elements of $\M_d$ (up to permutation of rows and columns) is $2^{O(d^3)}$. 
\end{thm}

\begin{proof}
Let $S(A,B)$ be the slack matrix for a $2$-level $d$-configuration $(A,B)$. In our proof, we show that the matrix $S(A,B)$ (up to permutation of rows and columns) can be reconstructed from a set of at most $d(d+1)/2$ linear equations, where each linear equation belongs to a fixed collection of $2^d$ possible linear equations. This gives us the desired upper bound.

By Lemma~\ref{lem:01-configuration-B}, we can assume $B\subseteq\{0,1\}^d$. From the maximality of $(A,B)$, we get
\begin{align*}
 A&{}= \set{a\in \R^d}{\ip{a}{b}\in \{0,1\}\text{ for every } b\in B}\\
 & = \set{a\in \R^d}{ \ip{a}{b}^2-\ip{a}{b}=0\text{ for every } b\in B}\\
 & = \set{a\in \R^d}{\ip{(a a\tra ,-a)}{(bb\tra ,b)}=0\text{ for every } b\in B}\,.   
\end{align*}
Thus, knowing the linear space spanned by the vectors $(bb\tra ,b)$, $b\in B$, we are able to first reconstruct the set $A$ and then reconstruct the set $B = \set{b \in \R^d}{\ip{a}{b} \in \{0,1\}\text{ for every } a \in A}$.

The dimension of the linear space span of all the vectors $(bb\tra ,b)$, $b\in \{0,1\}^d$ is $d(d+1)/2$. Since $B\subseteq \{0,1\}^d$, the dimension of the linear span of the vectors $(bb\tra ,b)$, $b\in B$ is at most $d(d+1)/2$. Thus, there is $J\subseteq B$, $|J|\leq d(d+1)/2$, such that the linear spaces spanned by $(bb\tra ,b)$, $b\in B$ and $(bb\tra ,b)$, $b\in J$ are the same. 
Hence, to define the desired linear subspace, it is enough to select at most $d(d+1)/2$ vectors $b\in B$. Each of these vectors is chosen in $\{0,1\}^d$. The result follows.
\end{proof}

\section{Second upper bound for 2-level cones}
\label{sec:upper_bound_cones}

In this section, we refine Theorem~\ref{thm:naive_upper_bound} for $2$-level configurations $(A,B)$ such that $A$ is a set of $0/1$-vectors and $B$ a set of integer vectors. Together with Lemmas~\ref{lem:01-configuration-polytope} and \ref{lem:01-configuration-cone}, our refinement implies Theorem~\ref{thm:second_upper_bound} from the introduction.

\begin{thm}\label{thm:second_upper_bound_conf}
The total number of $0/1$-matrices (up to permutation of rows and columns) which are slack matrices of some $2$-level $d$-configuration $(A,B)$ with $A\subseteq \{0,1\}^d$ and $B\subseteq \Z^d$ is at most $2^{O(d^2\log d)}$.
\end{thm}

\begin{proof}
The proof is analogous to the proof of Theorem~\ref{thm:naive_upper_bound}. Due to maximality of $(A,B)$, we have
\begin{align*}
 A&{}=\set{a\in \R^d}{\ip{a}{b}\in \{0,1\} \text{ for every } b\in B}\\
 & = \set{a\in \R^d}{ \ip{a}{b}^2-\ip{a}{b}=0 \text{ for every } b\in B}\\
 & = \set{a\in \R^d}{\ip{(bb\tra ,-b)}{(a a\tra ,a)}=0 \text{ for every } b\in B}\,.   
\end{align*}

Notice that for any $x\in\R^d$, we have
\[
\ip{(bb\tra ,-b)}{(xx\tra ,x)}=\ip{b}{x}^2-\ip{b}{x}=\ip{b}{x}(\ip{b}{x}-1)\,.
\]
Now if both $b$ and $x$ are integer vectors, we have $\ip{b}{x}\in \Z$ and thus
\(
\ip{(bb\tra ,-b)}{(x x\tra ,x)} \geq 0
\).
Hence, for every $b\in B$ the inequality $\ip{(bb\tra ,-b)}{z}\ge 0$ defines a face of the  cone
\begin{equation}\label{eq:K_correlation}
K\coloneqq \cone\big(\set{z\in \R^{d^2+d}}{z=(xx\tra ,x) \text{ for some } x\in \{0,1\}^d}\big)\,.
\end{equation}
We remark that, since $x \in \{0,1\}^d$, we have $x_i^2 = x_i$ for all $i$. Thus, the vector $x$ appears on the diagonal of the matrix $xx\tra$. We deduce that the cone $K$ in~\eqref{eq:K_correlation} is linearly equivalent to the \emph{correlation cone}~\cite[Chapter 5]{Deza09}, which is defined as the conic hull of all matrices $xx\tra \in \R^{d^2}$ for $x \in \{0,1\}^d$.

The set $A$ is the set of all $x\in \{0,1\}^d$ such that the vector $z=(xx\tra ,x)$ lies in the following face of $K$:
\[
F \coloneqq \set{z\in K}{\ip{(bb\tra ,-b)}{z}=0 \text{ for every } b\in B}\,.
\]

Note that $K$ is a pointed cone of dimension $d(d+1)/2$, hence each of its faces is uniquely defined by the sum of at most $d(d+1)/2$ rays from the set of its extreme rays
\[
\set{z \in F}{ z=(xx\tra ,x) \text{ for some } x\in \{0,1\}^d}\,.
\]
Such a sum is always an integer vector in $[0,d(d+1)/2]^{d^2+d}$. Hence, the total number of possible faces of $K$ is at most $2^{O(d^2\log d)}$. Moreover, since each face defines at most one possible $A\subseteq\{0,1\}^d$, we have the desired upper bound on the total number of different sets $A\subseteq\{0,1\}^d$, and hence on the total number of different $2$-level $d$-configurations $(A,B)$, where $A\subseteq\{0,1\}^d$ and $B\subseteq \Z^d$.
\end{proof}

\section{Second upper bound for 2-level configurations}
\label{sec:improved_upper_bound}

In this section, we improve the upper bound in Theorem~\ref{thm:naive_upper_bound} for general $2$-level configurations, that is, without the extra assumptions of the previous section. This is our final result. It implies Theorem~\ref{thm:refined_upper_bound} from the introduction.

\begin{thm}\label{thm:refined_upper_bound_slack_matrices}
The number of maximal elements of $\M_d$ (up to permutation of rows and columns) is $2^{O(d^2\log^3(d))}$.
\end{thm}

\begin{proof}
Let $S(A,B)$ be the slack matrix of a $2$-level $d$-configuration $(A,B)$. By Lemma~\ref{lem:01-configuration-B}, we may assume that $B\subseteq \{0,1\}^d$ and $A\subseteq\R^d$.

Let us show that there are vectors $b_1,\ldots, b_k \in B$, $k\leq d + d \log d$, such that the lattice $\Lambda(B)$ generated by $B$ equals the lattice $\Lambda(\{b_1,\ldots,b_k\})$ generated by $b_1$, \ldots, $b_k$. Let us start with a set of $d$ linearly independent vectors $C\coloneqq \{b_1,\ldots,b_d\}$ from $B$. Now until $\Lambda(C)=\Lambda(B)$, we iteratively replace $C$ by $C\cup \{b\}$ for some $b\in B\setminus \Lambda(C)$. 

Note that if $\Lambda(C)\neq\Lambda(B)$, then the lattice $\Lambda(C)$ is a proper sublattice of $\Lambda(C\cup \{b\})$ for $b\in B\setminus \Lambda(C)$. Thus the determinant of $\Lambda(C)$ equals the determinant of 
$\Lambda(C\cup \{b\})$ times an integer strictly larger than~$1$. Since $\{b_1,\ldots,b_d\} \subseteq B \subseteq \{0,1\}^d$, the determinant of the initial lattice $\Lambda(\{b_1,\ldots, b_d\})$ is at most $d^d = 2^{d\log d}$. Whenever $\Lambda(C) \neq \Lambda(B)$, the determinant of $\Lambda(C)$ decreases each time by an integer factor larger than or equal to~$2$. Thus, by the time $\Lambda(C)=\Lambda(B)$, the cardinality of $C$ is at most $d + d \log d$.

Now we define maps $\zeta : A \rightarrow \{0,1\}^k$ and $\phi : B\rightarrow \Z^k$ such that 
\[
\ip{a}{b}=\ip{\zeta(a)}{\phi(b)} \quad \text{ for every } a \in A,\ b \in B\,.
\]
For $a \in A$, we let $\zeta(a)\coloneqq (\ip{a}{b_1},\ldots, \ip{a}{b_k})$. For $b \in B$, we let $\phi(b)\coloneqq (\lambda_1,\ldots,\lambda_k)$, where $\lambda_i\in \Z$, $i=1,\ldots,k$ are integer coefficients verifying $\lambda_1 b_1 + \ldots + \lambda_k b_k=b$. For $i = 1, \ldots, k$, we let $\phi(b_i) \coloneqq \ee_i$.

Thus
\[
\ip{\zeta(a)}{\phi(b)} = \sum_{i=1}^k \ip{a}{b_i} \lambda_i = \ip{a}{\sum_{i=1}^k \lambda_i b_i} = \ip{a}{b}\,.
\]

Let $A' \coloneqq \set{a' \in \R^k}{\ip{a'}{b'} \in \{0,1\} \text{ for all } b' \in \phi(B)}$. Then $\zeta(A) \subseteq A' \subseteq \{0,1\}^k$. Now, similarly to the proof of Theorem~\ref{thm:second_upper_bound_conf}, one can show that there are at most $2^{O(k^2\log k)}$ possible $A'$. The map $\zeta$ is uniquely defined by the vectors $b_1, \ldots, b_k\in B\subseteq\{0,1\}^d$. Thus, there are at most $2^{kd}$ possible maps $\zeta$. Note that, by extending $\zeta$ in the obvious way to a map defined on the whole $\R^d$,
\[
A=\set{a\in \R^d}{\zeta(a)\in A'}\,,
\]
showing that there are at most $2^{O(k^2\log k+kd)}$, and so at most $2^{O(d^2\log^3 d)}$, possibilities for $A$.
\end{proof}

\section{Discussion}

Similarly to Theorem~\ref{thm:naive_upper_bound}, one can show that the total number of maximal matrices of rank $d$ with entries in $\{0,1,\ldots,k\}$ (and no repeated row or column) is at most $(k+1)^{d\binom{d+k+1}{k}}$.

This gives an upper bound of $(k+1)^{d\binom{d+k+1}{k}}$ on the number of linear equivalence classes of $d$-cones $K= \set{x \in \R^d}{Ax \geqslant \0} = \cone(V)$ such that $A_i v_j \in \{0,1,\ldots,k\}$ for all $i\in [m]$ and  $j\in [n]$, where $A\in\R^{m\times d}$ has $m$ rows denoted by $A_i$, $i \in [m]$ and $V\subseteq \R^d$ has $n$ vectors denoted by $v_j$, $j \in [n]$.

Also, this gives an upper bound of $(k+1)^{(d+1)\binom{d+k+2}{k}}$ on the number of affine equivalence classes of $d$-polytopes $P = \set{x \in \R^d}{Ax \ge b} = \conv(V)$ such that $A_i v_j-b_i \in \{0,1,\ldots,k\}$ for all $i\in [m]$ and $j\in [n]$, where $A\in\R^{m\times d}$ has $m$ rows denoted by $A_i$, $i \in [m]$ and $V\subseteq \R^d$ has $n$ points denoted by $v_j$, $j \in [n]$.

We would like to point out that there are infinitely many affine equivalence classes of polytopes $P = \set{x \in \R^d}{Ax \geqslant b} = \conv(V)$ such that 
\[
\big|\set{A_i v_j-b_i \in \{0,1,\ldots,k\}}{i\in [m]\text{ and } j\in [n]}\big|\le k
\]
already for $d=2$ and $k=3$. This is due to the observation that every quadrilateral $P\subseteq \R^2$ has a pair of outer and inner descriptions satisfying the above condition for $k=3$. However, there are infinitely many affine equivalence classes of quadrilaterals.

We leave it as an open problem to fill the gap between the lower bound of $2^{\Omega(d^2)}$ and the upper bound of $2^{O(d^2\log d)}$ for the number of affinely inequivalent $2$-level $d$-polytopes and the number of linearly inequivalent $2$-level $d$-cones, and also find better estimates on the number of linear equivalence classes of $2$-level configurations.

\section*{Acknowledgments}

We acknowledge support from ERC grant \emph{FOREFRONT} (grant agreement no. 615640) funded by the European Research Council under the EU's 7th Framework Programme (FP7/2007-2013).

This work was done while the authors were visiting the Simons Institute for the Theory of Computing. It was partially supported by the DIMACS/Simons Collaboration on Bridging Continuous and Discrete Optimization through NSF grant $\#$CCF-1740425.

\bibliographystyle{amsplain}
\bibliography{bibliography}

\end{document}